\documentclass[12pt,a4paper]{amsart}

\usepackage{amsmath,amsfonts,amsthm,amssymb,verbatim,enumerate,graphicx}
\usepackage{color}
\usepackage{hyperref}
\usepackage[marginratio=1:1,tmargin=117pt, height=600pt]{geometry}

\usepackage[flushmargin]{footmisc}
\usepackage{pdflscape}
\newtheorem{theorem}{Theorem}[section]

\newtheorem{proposition}[theorem]{Proposition}

\newtheorem{example}{Example}
\newtheorem{question}{Question}

\theoremstyle{remark}
\newtheorem*{remark}{Remark}

\theoremstyle{definition}
\newtheorem{definition}[theorem]{Definition}

\newcommand{\tef}{transcendental entire function}
\newcommand\qfor{\quad\text{for }}
\newcommand \C{\mathbb{C}}
\newcommand \N{\mathbb{N}}
\newcommand \R{\mathbb{R}}
\def\seq{(z_n)_{n\geq 0}}
\newcommand*{\defeq}{\mathrel{\vcenter{\baselineskip0.5ex \lineskiplimit0pt
                     \hbox{\scriptsize.}\hbox{\scriptsize.}}}%
                     =}
\begin{document}
%
%
%
%
\title[Which sequences are orbits?]{Which sequences are orbits?}
\author[{D. A. Nicks \and D. J. Sixsmith}]{Daniel A. Nicks \and David J. Sixsmith}
\address{School of Mathematical Sciences \\ University of Nottingham \\ Nottingham
NG7 2RD \\ UK \newline \indent \href{https://orcid.org/0000-0002-9493-2970}{\includegraphics[width=1em,height=1em]{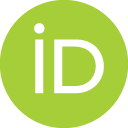} {\normalfont https://orcid.org/0000-0002-9493-2970}}}
\email{Dan.Nicks@nottingham.ac.uk}
\address{Department of Mathematical Sciences \\
	 University of Liverpool \\
   Liverpool L69 7ZL\\
   UK \newline \indent \href{https://orcid.org/0000-0002-3543-6969}{\includegraphics[width=1em,height=1em]{orcid2.png} {\normalfont https://orcid.org/0000-0002-3543-6969}}}
\email{djs@liverpool.ac.uk}
\thanks{\vspace{3pt}\vspace{3pt}\\ }
%
%
%
%
\begin{abstract}
In the study of discrete dynamical systems, we typically start with a function from a space into itself, and ask questions about the properties of sequences of iterates of the function. In this paper we reverse the direction of this study. In particular, restricting to the complex plane, we start with a sequence of complex numbers and study the functions (if any) for which this sequence is an orbit under iteration. This gives rise to questions of existence and of uniqueness. We resolve some questions, and show that these issues can be quite delicate.
\end{abstract}
\maketitle
%
%
%
\section{Introduction}
Complex dynamics usually begins with a function $f : X \to X$, for some Riemann surface $X$. For a point $z \in X$, we then consider the properties of the \emph{orbit} of $z$; in other words, the sequence of images under iteration of $f$
 $$z, f(z), f^2(z) \defeq f(f(z)), \ldots.$$

Our goal in this paper is to reverse these considerations, at least in the slightly restricted case where $X = \C$. In other words, we begin with a sequence $\seq$ of elements of $\C$, and then ask the following questions.
\begin{enumerate}[(I)]
\item Does there exist a function $f$ that \emph{realises} the sequence? In other words, for each $n \in \N$, we have that $f^n(z_0) = z_n$.\label{existence}
\item If $\seq$ is realised by some function $f$, then is $f$ unique with this property?\label{uniqueness}
\item Do the answers to \eqref{existence} and \eqref{uniqueness} change if we restrict to different classes of functions? For example, we might restrict to polynomials, entire functions, meromorphic functions or quasiregular functions.
\end{enumerate}
It is immediately clear that not every sequence can be an orbit; it is easy to see, for example, that the sequence $1, 2, 1, 3, \ldots$ cannot be realised by any function since the point $1$ cannot map both to $2$ and to $3$. We will show that the following definition is natural.
\begin{definition}
\label{def:candidateorbit}
Suppose that $\seq$, is a sequence of complex numbers. We say that $\seq$ is a \emph{candidate orbit} if the following condition holds. Suppose that $z \in \C$ and that $(n_j)_{j\in\N}$ is a sequence of non-negative integers such that $z_{n_j} \rightarrow z$ as $j\rightarrow\infty$. Then there is a point $z' \in \C$, \emph{which depends only on $z$}, such that $z_{n_j+1} \rightarrow z'$ as $j \rightarrow\infty$.
\end{definition}
\begin{remark}\normalfont
Note that in this definition we allow the sequence $(n_j)_{j\in\N}$ to be constant.
In particular, if $\seq$ is a candidate orbit, then $z_p = z_q$ implies that $z_{p+k} = z_{q+k}$ for $k \in \N$.
\end{remark}
Our first result justifies the definition of a candidate orbit, and also gives a complete answer to \eqref{existence} and \eqref{uniqueness} in the case of \emph{continuous} functions.
\begin{theorem}
\label{theo:cts}
Suppose that $\seq$ is a sequence of complex numbers. Then $\seq$ is a candidate orbit if and only if there is a continuous map $f : \C \to \C$ that realises $\seq$. Moreover, $f$ is unique with this property if and only if the set $\{ z_n : n \geq 0 \}$ is dense in $\C$.
\end{theorem}
Next we consider the case of entire functions. While Theorem~\ref{theo:cts} gives a complete answer for continuous functions, it is clear that
the condition of being a candidate orbit is far from sufficient for it to be realised by an entire function.
In this case we are able to give a complete answer to the uniqueness question, but we do not fully resolve the question of existence.

We require some additional definitions. First, we say that a candidate orbit $\seq$ is \emph{periodic} if there exist $n \ne n' \in \N$ such that $z_n = z_{n'}$; note that this definition includes sequences that are often called ``preperiodic''. Second, we say that $\seq$ is \emph{escaping} if $z_n \rightarrow\infty$ as $n\rightarrow\infty$. Finally, $\seq$ is \emph{bounded} if there exists $L>0$ such that $|z_n| \leq L$, for all $n \geq 0$, and it is \emph{bungee} if it is not bounded and not escaping. 

We then have the following, where we distinguish carefully between polynomials and transcendental entire functions.
\begin{theorem}
\label{theo:uniq}
Suppose that $\seq$ is a candidate orbit. Then exactly one of the following holds.
\begin{enumerate}[(a)]
\item $\seq$ is periodic, and is realised by infinitely many {\tef}s and infinitely many polynomials.\label{cond:per}
\item $\seq$ is escaping, and is realised by infinitely many {\tef}s and at most one polynomial. \label{cond:esc}
\item $\seq$ is bungee, and is realised by at most one entire function and no polynomials.\label{cond:unotesc}
\item $\seq$ is bounded and not periodic, and is realised by at most one entire function.\label{cond:bnotper}
\end{enumerate}
\end{theorem}
\begin{remark}\normalfont
An interesting implication of this theorem is the following, which also follows from the Identity theorem. It is now common in complex dynamics, see, for example, \cite{OsSix}, to classify orbits into three type; \emph{escaping}, \emph{bounded}, and \emph{bungee}. Suppose we are given, as data, a bungee orbit of a {\tef}, but are not given the function; note that a polynomial cannot have a bungee orbit, whereas these are always present for {\tef}s, see \cite{OsSix}. Then Theorem~\ref{theo:uniq} implies that this orbit completely determines the {\tef}, and so from this orbit alone we can, in principle, deduce all the other dynamics of the function. This is never true of an escaping orbit, and is true for a bounded orbit if and only if it is not periodic.
\end{remark}

Theorem~\ref{theo:uniq} gives a complete result for {\tef}s apart from the following question.
\begin{question}
\label{question:1}
If $\seq$ is a candidate orbit that is either bungee, or bounded but not periodic, does there exist a {\tef} that realises $\seq$?
\end{question}

In general, this question seems difficult. Note that the sequences considered in Question~\ref{question:1} are exactly those that have a finite accumulation point. In order to provide terminology required for our result concerning candidate orbits that have a finite accumulation point, we first prove the following.

\begin{proposition}
\label{prop:interpol}
Suppose that $(\zeta_n)_{n \geq 0}$ is a sequence of points of $\C$, tending to a point $\zeta$, and with $\zeta_n \ne \zeta$ for all $n$. Suppose also that $(w_n)_{n \geq 0}$ is a sequence of points of $\C$, tending to a point $w$. Suppose finally that $U$ is a disc centred at $\zeta$ containing all the $\zeta_n$ and that there is a function $f$, analytic on $U$, such that $f(\zeta_n) = w_n$, for $n \geq 0$. Then the following limits (defined iteratively) all exist:
\begin{align*}
p &\defeq \lim_{n\to\infty} \frac{\log|w_n - w|}{\log|\zeta_n - \zeta|}; \\
a_p &\defeq \lim_{n\to\infty} \frac{w_n - w}{(\zeta_n-\zeta)^p};         \\
a_{p+k} &\defeq \lim_{n\to\infty} \frac{(w_n - w) - \sum_{j=0}^{k-1} a_{p+j} (\zeta_n-\zeta)^{p+j}}{(\zeta_n-\zeta)^{p+k}}, \text{ for } k \geq 1;
\end{align*}
where $p \in \N$ and $a_p \ne 0$.
Moreover, the formal Taylor series
\begin{equation}
\label{fdef}
f(z) = w + a_p (z-\zeta)^p + a_{p+1} (z-\zeta)^{p+1} + \ldots,
\end{equation}
converges on $U$.
\end{proposition}

Our result giving a necessary condition for Question~\ref{question:1} is the following immediate consequence of Proposition~\ref{prop:interpol}.
\begin{theorem}
\label{theo:accum}
Suppose that $\seq$ is realised by a \tef~$f$, and that $z_{n_j}$ is a convergent subsequence tending to a finite point $\zeta$ with $z_{n_j} \ne \zeta$ for all $j$. Then the limits in Proposition~\ref{prop:interpol} all exist, with $\zeta_j = z_{n_j}$, $w_j = z_{n_j+1}$ and $w = f(\zeta)$, and $f$ has a Taylor series expansion about $\zeta$ as given in \eqref{fdef} that converges on $\C$.
\end{theorem}

Finally, we restrict to the case that $\seq$ is a candidate orbit with a unique finite accumulation point $\zeta$. Observe that, by considering instead the sequence $(z_n - \zeta)_{n\geq 0}$, we can assume that $\zeta=0$. It follows from Theorem~\ref{theo:accum} that there are very strong constraints for a sequence tending to zero to be realised by an analytic function; in view of the Identity theorem this is perhaps unsurprising. We discuss some examples of candidate orbits tending to zero that cannot be realised by any analytic function in Section~\ref{S.transcendental} below.


It is natural to ask, therefore, if these constraints can be relaxed if we allow the function to be \emph{quasiregular}. Roughly speaking, a quasiregular map is a map from $\C$ to $\C$ that is continuous, differentiable almost everywhere, and has uniformly bounded distortion; see \cite{Rickman, Vuorinen} for a more precise definition. A \emph{quasiconformal} map is a quasiregular map that is also a homeomorphism.

Our final results show that in this setting the constraints on the sequence can be significantly weaker. First we give a necessary condition for a sequence tending to zero to be realised by a quasiregular map.
\begin{theorem}
\label{theo:qrnec}
Suppose that $\seq$ is a sequence of points of $\C$ that tends to zero. If $\seq$ is realised by a quasiregular map, then there exist $\mu, \nu > 0$, $C>1$, and $n_0 \in \N$ such that
\begin{equation}
\label{e1}
\frac{1}{C^2} \left(\frac{|z_n|}{|z_{n+1}|}\right)^\mu \leq \frac{|z_{n+1}|}{|z_{n+2}|} \leq C^2 \left(\frac{|z_n|}{|z_{n+1}|}\right)^\nu, \ \text{for } n \geq n_0 \ \text{such that } |z_n| \geq |z_{n+1}|,
\end{equation}
and
\begin{equation}
\label{e1b}
\frac{1}{C^2} \left(\frac{|z_{n+1}|}{|z_{n}|}\right)^\mu \leq \frac{|z_{n+2}|}{|z_{n+1}|} \leq C^2 \left(\frac{|z_{n+1}|}{|z_{n}|}\right)^\nu, \ \text{for } n \geq n_0 \ \text{such that } |z_n| \leq |z_{n+1}|.
\end{equation}
\end{theorem}
\begin{remark}\normalfont
It seems natural to ask if the inequalities \eqref{e1} and \eqref{e1b} are equivalent to the condition that there exist $\alpha, \beta > 0$ such that
\begin{equation}
\label{eq:holder}
|z_n|^\alpha \leq |z_{n+1}| \leq |z_n|^\beta, \quad \text{for all sufficiently large values of } n.
\end{equation}
We show in the Appendix that slightly weaker conditions than \eqref{e1} and \eqref{e1b} do indeed imply \eqref{eq:holder}, but we give an example to show that the reverse implication does not hold.
\end{remark}
The following result gives a sufficient condition for a sequence to be realised by a quasiregular map.
\begin{theorem}
\label{theo:qrsuf}
Suppose that $\seq$ is a sequence of points of $\C$, tending to zero and of strictly decreasing modulus. Suppose also that there exist $\mu, \nu, n_0 > 0$ and $C>1$ such that \eqref{e1} holds. Suppose finally that there exists $D \in (0, 1)$ such that
\begin{equation}
\label{eq:condition2}
|z_{n+1}| \leq D |z_n|, \qfor n \geq 0.
\end{equation}
Then $\seq$ is realised by a quasiconformal map of $\C$.
\end{theorem}
\begin{remark}\normalfont
It is natural to ask if it is possible to omit the condition \eqref{eq:condition2}, provided that the other conditions of Theorem~\ref{theo:qrsuf} are satisfied. We give two examples below to illustrate the difficulties involved in such questions. Example~\ref{ex:qrcleverexample} gives such a sequence which is, in fact, realised by a quasiregular map, but which cannot be realised using the techniques used in the proof of Theorem~\ref{theo:qrsuf}. Example~\ref{ex:qrnew} is a sequence of strictly decreasing modulus, which satisfies \eqref{e1}, but which we show cannot be realised by any quasiregular map. It seems that giving both necessary and sufficient conditions for a sequence which tends to zero to be realised by a quasiregular map is a difficult problem.
\end{remark}

Finally, we observe that that is easy to see how Definition~\ref{def:candidateorbit} can be modified to apply to a sequence of points of $\R^d$ for $d \geq 1$. It is equally straightforward to modify Theorem~\ref{theo:cts} and Theorem~\ref{theo:qrnec} to apply in this more general setting.
\section{Continuous functions}
\label{S.continuous}
This section is dedicated to the proof of Theorem~\ref{theo:cts}, which requires the Tietze extension theorem; see, for example, \cite[Theorem 15.8]{Willard}.
\begin{theorem}
\label{theo:tietze}
Suppose that $X$ is a normal topological space, that $A \subset X$ is closed, and that $f : A \to \R$ is continuous. Then there exists a continuous map $F : X \to \R$ that agrees with $f$ on $A$.
\end{theorem}
\begin{proof}[Proof of Theorem~\ref{theo:cts}]
In one direction, suppose that $f : \C \to \C$ is a continuous map that realises the sequence $\seq$. Suppose that $(n_j)_{j\in\N}$ is a sequence of natural numbers such that the sequence $z_{n_j}$ tends to a point $z \in \C$. It follows by the continuity of $f$ that the sequence $z_{n_j+1} = f(z_{n_j})$ tends to a limit $z'$ such that $z' = f(z)$. This completes the proof in this direction.

In the other direction, suppose that $\seq$ is a candidate orbit. Set $$S \defeq \{ z_n : n \geq 0 \}.$$ We begin by constructing a function $f : \overline{S} \to \C$. First we define $f$ on $S$ by setting
\[
f(z_{n-1}) = z_n, \qfor n \in \N.
\]
Note that the fact that $\seq$ is a candidate orbit implies that this choice is well-defined.

Now suppose that $z \in \overline{S} \setminus S$. Then there is a sequence $(n_j)_{j\in\N}$ such that $z_{n_j} \rightarrow z$ as $j\rightarrow\infty$. By the definition of a candidate orbit, there is a point $z' \in \C$, which depends only on $z$, such that $z_{n_j+1} \rightarrow z'$ as $j \rightarrow\infty$. We define $f(z) = z'$.
Note that this function is well-defined because of the assumption in the definition of a candidate orbit that $z'$ depends only on $z$.

Next, we claim that for any $z\in \overline{S}$ and $\epsilon >0$ there exists $z_n\in S$ such that $|z-z_n|<\epsilon$ and $|f(z)-f(z_n)|<\epsilon$. This is clear if $z\in S$; we pick $z_n=z$. If $z \in \overline{S} \setminus S$, then there is a sequence  $z_{n_j} \rightarrow z$ and by definition $$f(z) = \lim_{j\to\infty} z_{n_j +1} = \lim_{j\to\infty} f(z_{n_j}),$$ so the claim follows by taking $n=n_j$ with $j$ large.

We now prove that $f$ is continuous on $\overline{S}$. To this end, we take a sequence of points $(w_p)_{p\in\N}$ in $\overline{S}$ tending to a point $w \in \overline{S}$ and aim to prove that $f(w_p) \to f(w)$ as $p\to\infty$. By the claim above, we can find $z_{n_p}\in S$ such that
\begin{equation}
\label{eq:2seqsconv}
 |w_p - z_{n_p}| < \frac1p \quad \mbox{ and } \quad   |f(w_p) - f(z_{n_p})| < \frac1p, \qfor p \in \N.
\end{equation}
In particular, $\lim_{p\to\infty} z_{n_p} = \lim_{p\to\infty} w_p = w$. Hence, by definition and construction, $\lim_{p\to\infty} f(z_{n_p}) = f(w)$; to see this in the case that $w\in S$, say $w=z_N$, we use the fact that the constant sequence $(z_N)_{p\in\N}$ has the same limit as $(z_{n_p})_{p\in\N}$ and thus by Definition~\ref{def:candidateorbit},
\[ \lim_{p\to\infty} f(z_{n_p}) = \lim_{p\to\infty} z_{n_p +1} = \lim_{p\to\infty} z_{N+1} = f(w). \]
It now follows from \eqref{eq:2seqsconv} that $f(w_p) \to f(w)$ as $p\to\infty$, proving continuity on $\overline{S}$.

The first part of the result then follows by applying the Theorem~\ref{theo:tietze} first to the real part of $f$ in $\overline{S}$, and then to the imaginary part of $f$ in $\overline{S}$.

The claim regarding uniqueness can be seen to be true as follows. First we note that the construction of $f$ in $\overline{S}$ was clearly unique. Hence, if $S$ is dense in the plane then $f$ is unique. However, if $S$ is not dense in the plane, then prior to applying Theorem~\ref{theo:tietze} we can choose a point $w \notin \overline{S}$ and any point $z \in \C$, and fix $f(w) = z$. We may then apply Theorem~\ref{theo:tietze} to the set $\overline{S} \cup \{w\}$, and obtain infinitely many continuous functions that realise the candidate orbit.
\end{proof}
\section{Entire functions}
\label{S.transcendental}
\begin{proof}[Proof of Theorem~\ref{theo:uniq}]
First we prove \eqref{cond:per}, and so we suppose that $\seq$ is a periodic candidate orbit. Let $n$ be the least integer such that $z_{n+1} = z_{n'}$ for some $0 \leq n' \leq n$. We need to find functions such that the orbit of $z_0$ begins $z_0, z_1, \ldots, z_n, z_{n'}$. For convenience, denote these numbers by $w_0, w_1, \ldots, w_n, w_{n+1}$.

Let $P$ be the polynomial
\[
P(z) \defeq \prod_{k=0}^{n} (z - w_k).
\]

Let $F$ be any polynomial if we are trying to find a polynomial, or any {\tef} if we are trying to find a transcendental entire function. Consider the function
\begin{equation}
\label{Fdef}
f(z) \defeq \sum_{k=0}^{n} \left(\frac{P(z)}{(z-w_k)} \cdot \left(F(z) - F(w_k) + \frac{w_{k+1}}{\prod_{k'=0, k' \ne k}^{n} (w_k - w_{k'})}\right)\right).
\end{equation}
This function has the required properties. (Note that this part of the result can also be proved using the same technique as used later for escaping sequences, but the explicit nature of \eqref{Fdef} is appealing.)

Next we prove \eqref{cond:bnotper}. Let $\seq$ be a bounded candidate orbit that is not periodic and suppose that $\seq$ is realised by two entire functions $f$ and $g$. Then the set $\{ z_n : n\ge 0\}$ is bounded and infinite and so accumulates at some point $z\in\C$. But each $z_n$ is a zero of the entire function $f-g$, and therefore $f\equiv g$.

The case \eqref{cond:unotesc} is almost identical to \eqref{cond:bnotper}, and we omit the details. It is easy to see that no polynomial can realise a sequence of this type.

To prove case \eqref{cond:esc}, suppose that $\seq$ is an escaping candidate orbit. We first show that there is a {\tef} $f$ that realises this orbit; note that this is essentially \cite[Exercise 6, p.26]{Knopp}. By the Weierstrass factorization theorem, there exists a {\tef} $f_0$ such that $f_0(z_n) = 0$, for $n \geq 0$, and all these zeros are simple. By Mittag-Leffler's theorem, there exists a transcendental meromorphic function $f_1$ with simple poles at the points $z_n$, with residues at these points equal to $z_{n+1}/f_0'(z_n)$, and with no other poles. Then the {\tef} $f(z) \defeq f_0(z) f_1(z)$ realises the orbit $\seq$. To see that there are many such functions, choose any entire function $h$ and set $g(z) \defeq f_0(z)\cdot (f_1(z) + h(z))$. Then $g$ is a {\tef} that realises $\seq$.

Note finally that $\seq$ cannot be realised by two polynomials, $P, Q$ say, as otherwise the non-constant polynomial $P-Q$ has infinitely many zeros. Here we make the additional remark that $\seq$ can only be realised by a polynomial if $\log|z_{n+1}| / \log|z_n|$ tends to an integer as $n\to\infty$.
\end{proof}

\begin{proof}[Proof of Proposition~\ref{prop:interpol}]
Suppose that the hypotheses of the proposition all hold. We must have $f(\zeta) = w$. Since $f$ is analytic in a neighbourhood of $\zeta$, there is a Taylor series for $f$ about $\zeta$ of the form \eqref{fdef}. Since $f(\zeta_n) = w_n$, for $n \geq 0$, we obtain
\begin{equation}
\label{zneq}
w_{n} = w + a_p (\zeta_n-\zeta)^p + a_{p+1} (\zeta_n-\zeta)^{p+1} + \ldots.
\end{equation}
The statements regarding the limits are then consequences of \eqref{zneq}, since $\zeta_n \rightarrow \zeta$ as $n \rightarrow \infty$.
\end{proof}


As promised in the introduction, we end this section with some examples of apparently straightforward candidate orbits, tending to zero, which cannot be realised by an analytic function.
\begin{example}\normalfont
Consider the candidate orbit $1, \frac{1}{2}, \frac{1}{4}, \frac{1}{16}, \frac{1}{256}, \ldots$; each term, apart from the second, is the square of the one before. The limits obtained from Proposition~\ref{prop:interpol} give $f(z) = z^2$, unsurprisingly. Since $f(1) \ne \frac{1}{2}$, this sequence cannot be realised by an analytic function. \label{ex1}
\end{example}
\begin{example}\normalfont
In the previous example, the function ``fails'' at the first term, but does correctly generate the rest of the orbit. Let $\epsilon_n$ be a sequence of positive numbers that tends to zero much more quickly than $2^{-2^n}$; for example $\epsilon_n = 2^{-10^{n+1}}$. Then consider the candidate orbit
\[
\frac{1}{2} +  \epsilon_1, \frac{1}{4} + \epsilon_2, \frac{1}{16} + \epsilon_3, \frac{1}{256} + \epsilon_4, \ldots.
\]
Again the limits obtained from Proposition~\ref{prop:interpol} give $f(z) = z^2$, but $f$ does not generate any part of the orbit above.
\end{example}
\begin{example}\normalfont
Choose $1<a<2$ and set
\[
z_n \defeq 2^{-a^n}, \qfor n\geq 0.
\]
This sequence cannot be realised by an analytic function, since the first limit resulting from Proposition~\ref{prop:interpol} is not an integer.
\end{example}
\begin{example}\normalfont
The final example is more complicated. For each $m \geq 3$, let $N_m$ be a large integer, to be determined, and set
\[
\sigma_{m,n} \defeq
\begin{cases} 0, &\text{ for } 0 \le n < N_m, \\
              2^{m^2}, &\text{ otherwise}.
\end{cases}
\]
Then consider the candidate orbit given inductively by
\[
\begin{cases}
z_0 &\defeq \frac{1}{2}, \\
z_{n+1} &\defeq z_n^2 + \sigma_{3,n} z_n^3 + \sigma_{4,n} z_n^4 + \sigma_{5,n} z_n^5 + \ldots. 
\end{cases}
\]
By choosing the terms in the sequence $N_m$ to be sufficiently large, we can ensure that $\seq$ tends to zero, and indeed is very close to the sequence $\frac{1}{2}, \frac{1}{4}, \frac{1}{16}, \ldots$. However, the limits obtained from Proposition~\ref{prop:interpol} give rise to the formal power series
\[
f(z) = z^2 + 2^{3^2} z^3 + 2^{4^2} z^4 + 2^{5^2} z^5 + \ldots,
\]
the radius of convergence of which is zero.
\end{example}
These examples show that it is very difficult for a sequence tending to zero to be realisable by an analytic function. Note that Theorem~\ref{theo:qrsuf} shows that \emph{all} these sequences can be realised by quasiconformal maps.
\section{Quasiregular maps}
This section is dedicated to the proofs of Theorem~\ref{theo:qrnec} and Theorem~\ref{theo:qrsuf}, together with some comments on the hypotheses of those theorems. We require the following, which is a version of \cite[Theorem 1.1]{superattracting}; see also \cite[Lemma 3.9]{guletc}. Here for a quasiregular map $f : \C \to \C$, we denote the maximum modulus by $M(r) \defeq \max_{|x| = r} |f(x)|$, and the minimum modulus by $m(r) \defeq \min_{|x| = r} |f(x)|$.
\begin{theorem}
\label{theo1}
Suppose that $f : \C \to \C$ is quasiregular and non-constant with $f(0)=0$. Then there exist $C > 1$ and $\mu, \nu, r_0 > 0$ with the following property. If $T \in (0,1]$ and $r \in (0,r_0)$, then
\begin{equation}
\label{1.4}
T^{-\mu} \leq \frac{M(r)}{m(Tr)} \leq C^2 T^{-\nu},
\end{equation}
and
\begin{equation}
\label{1.5}
\frac{1}{C^2} T^{-\mu} \leq \frac{m(r)}{M(Tr)} \leq T^{-\nu}.
\end{equation}
\end{theorem}

\begin{proof}[Proof of Theorem~\ref{theo:qrnec}]
Suppose that $f : \C \to \C$ is a quasiregular map that realises $\seq$. 
Choose $n_0$ sufficiently large that $|z_n| < r_0$ for all $n\ge n_0$, where $r_0$ is the constant from Theorem~\ref{theo1}.
Suppose, for some $n \ge n_0$, we have $|z_n| \geq |z_{n+1}|$. Set $T = \frac{|z_{n+1}|}{|z_n|} \leq 1$.  We deduce from \eqref{1.4}, with $r = |z_n|$, that
\begin{equation*}
\frac{|z_{n+1}|}{|z_{n+2}|} \leq \frac{M(|z_n|)}{m(|z_{n+1}|)} \leq C^2 \left(\frac{|z_{n+1}|}{|z_n|}\right)^{-\nu}.
\end{equation*}
Similarly we deduce from \eqref{1.5} that
\begin{equation*}
\frac{1}{C^2} \left(\frac{|z_{n+1}|}{|z_n|}\right)^{-\mu} \leq \frac{m(|z_n|)}{M(|z_{n+1}|)} \leq \frac{|z_{n+1}|}{|z_{n+2}|}.
\end{equation*}
These inequalities give \eqref{e1}.

On the other hand, suppose, for some $n \ge n_0$, we have $|z_n| \leq |z_{n+1}|$. Set $T = \frac{|z_{n}|}{|z_{n+1}|} \leq 1$. We deduce from \eqref{1.4}, with $r = |z_{n+1}|$, that
\[
\frac{|z_{n+2}|}{|z_{n+1}|} \leq \frac{M(|z_{n+1}|)}{m(|z_{n}|)} \leq C^2 \left(\frac{|z_{n}|}{|z_{n+1}|}\right)^{-\nu}.
\]
Similarly we deduce from \eqref{1.5} that
\[
\frac{1}{C^2} \left(\frac{|z_{n}|}{|z_{n+1}|}\right)^{-\mu} \leq \frac{m(|z_{n+1}|)}{M(|z_{n}|)} \leq \frac{|z_{n+2}|}{|z_{n+1}|}.
\]
These inequalities give \eqref{e1b}.
\end{proof}
\begin{proof}[Proof of Theorem~\ref{theo:qrsuf}]
We first construct a model map, and it is slightly easier to do this in logarithmic coordinates. Suppose that $g : \R \to \R$ and $\theta : \R \to \R$ are differentiable. We define a map $\phi : \C \to \C$ by
\[
\phi(x+ iy) \defeq g(x) + i(y + \theta(x)).
\]
It is then a calculation that
\[
\frac{d\phi}{dz} = \frac{1}{2}\left(\frac{d\phi}{dx} - i\frac{d\phi}{dy}\right) = \frac{1}{2}\left(g'(x) + 1 + i\theta'(x)\right),
\]
and
\[
\frac{d\phi}{d\bar{z}} = \frac{1}{2}\left(\frac{d\phi}{dx} + i\frac{d\phi}{dy}\right) = \frac{1}{2}\left(g'(x) - 1 + i\theta'(x)\right).
\]
Hence the complex dilatation $\mu_{\phi} \defeq \dfrac{d\phi/d\bar{z}}{d\phi/{dz}}$ satisfies
\begin{equation}
\label{eq:mudef}
|\mu_{\phi}(x + iy)|^2 = \frac{(g'(x)-1)^2+\theta'(x)^2}{(g'(x)+1)^2+\theta'(x)^2}.
\end{equation}

For $\phi$ to be quasiregular, we require $|\mu_{\phi}|$ to be bounded strictly below one. In other words, we need there to exist $\epsilon > 0$ such that
\[
\frac{(g'(x)-1)^2+\theta'(x)^2}{(g'(x)+1)^2+\theta'(x)^2} \leq 1 - \epsilon.
\]
By a calculation this is equivalent to requiring that there exists $L> 0$ such that
\begin{equation}
\label{eq:bound}
\left(g'(x) + \frac{1}{g'(x)}\right) + \frac{\theta'(x)^2}{g'(x)} \leq L.
\end{equation}

We now focus on the special case that $g$ and $\theta$ are the linear maps given by $g(x)=(d'/d)x$ and $\theta(x)=(\alpha/d)x$, where $d, d' >0$ and $-4\pi \le \alpha \le 4\pi$. That is, we consider the real-linear map $\phi : \C \to \C$ of the form
\begin{equation}
\label{eq:phidef}
\phi(x+iy) = \frac{d'x}{d} + i\left(y + \frac{\alpha x}{d}\right).
\end{equation}
Since $\theta'(x)^2 = \alpha^2/d^2 \leq 16\pi^2/d^2$, we deduce from \eqref{eq:bound} that $\phi$ is $K$-quasiconformal where $K$ depends only on an upper bound for the quantity
\begin{equation}
\label{eq:bound2}
\frac{d'}{d} + \frac{d}{d'} + \frac{16\pi^2}{dd'}.
\end{equation}
In particular, for any $n\ge0$, if we take
\begin{equation}
\label{eq:phichoices}
d=\log \frac{|z_n|}{|z_{n+1}|}, \quad d'=\log \frac{|z_{n+1}|}{|z_{n+2}|} \quad \mbox{and} \quad \alpha = 2\arg z_{n+1} - \arg z_n - \arg z_{n+2},
\end{equation}
choosing the principle value of the argument in each case, then $d, d' \ge \log(1/D) >0$ by \eqref{eq:condition2}. Moreover, for $n\ge n_0$,  the assumption  \eqref{e1} gives that \eqref{eq:bound2} is bounded above by
\[
 \frac{1}{\mu} + \nu + \left( 2 + \frac{2}{\mu}\right)\frac{\log C}{\log(1/D)} + \frac{16\pi^2}{(\log (1/D))^2}.
\]
Since this bound is independent of $n$, there exists $K\ge1$ such that for any $n\ge0$ the map $\phi$ given by \eqref{eq:phidef} and \eqref{eq:phichoices} is $K$-quasiconformal.

Next we define a quasiconformal map $f : \C \to \C$, as follows. First we set $f(0) = 0$. For each $n \geq 0$, we define $f$ on the annulus
\[ A_n \defeq \{ z \in \C : |z_{n+1}| < z \leq |z_n| \}\]
by setting
\[
f(z) \defeq z_{n+2} \exp \left(\phi\left(\log\frac{z}{z_{n+1}}\right)\right), \qfor z \in A_n,
\]
where $\phi$ is as in \eqref{eq:phidef} with constants chosen as in \eqref{eq:phichoices}.
Note that $f$ is well-defined on each $A_n$ because, by definition, $\phi(z + 2\pi i) = \phi(z) + 2 \pi i$. In fact, $f$ is a $K$-quasiconformal map of $A_n$ onto $A_{n+1}$. It is not hard to check that $f(z_n e^{i\beta}) = z_{n+1} e^{i\beta}$ for $\beta\in\R$, from which it follows that $f$ realises the sequence $\seq$ and is continuous in $\{ z \in \C : |z| \le |z_0| \}$. Therefore we obtain our required quasiconformal map on $\C$ by setting $f(z) = z_1 z / z_0$ for $|z|>|z_0|$.
\end{proof}

As mentioned in the introduction, it is natural to ask if it is possible to omit a condition like \eqref{eq:condition2} from Theorem~\ref{theo:qrsuf}. To illustrate the difficulty of such questions, consider the following examples.
\begin{example}\normalfont
\label{ex:qrcleverexample}
Define a map $P : \C \to \C$ by choosing $\epsilon > 0$ and $s \in (0, 1)$ both small, and setting
\[
P(z) \defeq
\begin{cases}
|z|(1 - |z|)e^{i t(\arg z)}, &\quad\text{ for } |z| \leq s, \\
|z|(1 - s)e^{i t(\arg z)},   &\quad\text{ for } |z| > s,
\end{cases}
\]
where
\[t(y) \defeq
\begin{cases}
2y + \epsilon, &\quad\text{ for } 0\leq y \leq \pi/2, \\
\frac{2(y+\pi)}{3} + \epsilon,  &\quad\text{ for } \pi/2 < y < 2\pi,
\end{cases}
\]
and we choose the value of $\arg z$ in $[0, 2\pi)$. It can be checked that if $s$ is sufficiently small, then $P$ is a quasiconformal map of the plane.

Next choose $z_0 \in (0, s)$, and define the sequence $\seq$ iteratively by letting $z_{n+1} = P(z_n)$, for $n \geq 0$. Note that this sequence satisfies all the conditions of Theorem~\ref{theo:qrsuf}, apart from \eqref{eq:condition2}.

Clearly the sequence $\seq$ is realised by a quasiconformal map. However this sequence cannot be realised by a quasiconformal map constructed as in the proof of Theorem~\ref{theo:qrsuf}. To see this,  suppose that, in logarithmic coordinates, we have interpolated using the linear maps $g$ and $\theta$ exactly as in the Proof of Theorem~\ref{theo:qrsuf}.  Note that the ratio $|z_{n+1}|/|z_n| = 1 - |z_n|$ increases to~$1$ as $n \rightarrow \infty$. 
It can be checked that the derivative $g'(x)$ tends to $1$ as $x$ tends to zero. However, the absolute value of the derivative $\theta'(x)$ can be arbitrarily large and thus the left-hand side of \eqref{eq:bound} is not bounded above. This happens, for example, if $\epsilon$ is very small and $\arg z_n$ is small and positive. We can then calculate that $\arg z_{n+1} = 2 \arg z_n + \epsilon$ and $\arg z_{n+2} = 2\arg z_{n+1} + \epsilon$, so for suitable values of $x$ we obtain
\[
\theta'(x) = \frac{\arg z_{n}+\epsilon}{\log(1- |z_n|)},
\]
which can have very large absolute value for large values of $n$.
\end{example}
\begin{example}\normalfont
\label{ex:qrnew}
Define the sequence $\seq$ by, for $m \geq 0$,
\[
\begin{cases}
z_{3m}   &\defeq e^{-(m+2)}, \\
z_{3m+1} &\defeq e^{-(m+2)} - e^{-(m+2)^2}, \\
z_{3m+2} &\defeq -\frac{e^{-(m+2)}}{2}.
\end{cases}
\]
It can be checked that this is a candidate orbit which satisfies all the hypotheses of Theorem~\ref{theo:qrsuf} apart from \eqref{eq:condition2}. We show that it is not possible to realise this sequence with a quasiregular map.

Suppose, by way of contradiction, that $\seq$ is realised by map $f$ which is quasiregular on a neighbourhood, $U$, of the origin. It is known that quasiregular maps satisfy a H\"older condition; see, for example, \cite[Theorem III.1.11]{Rickman}. In other words, there exist $\alpha \in (0, 1]$, and $r, C > 0$ such that
\[
|f(x) - f(y)| \leq C |x-y|^\alpha, \qfor x, y \in \{ z \in \C : |z| < r \}.
\]
In particular, taking $x = z_{3m}$ and $y = z_{3m+1}$ yields, for large values of $m$, that
\[
\frac{3}{2}e^{-(m+2)} - e^{-(m+2)^2} \leq Ce^{-\alpha(m+2)^2}.
\]
This is impossible for sufficiently large values of $m$, completing the proof of our claim.
\end{example}
\section*{Appendix}
Our goal in this appendix is to show that the inequalities \eqref{e1} and \eqref{e1b} imply \eqref{eq:holder}, but that the reverse implication does not hold. First we have the following, which has slightly weaker hypotheses than \eqref{e1} and \eqref{e1b}; we repeat \eqref{eq:holder} for the convenience of the reader.
\begin{proposition}
Suppose that $(x_n)_{n\in\N}$ is a sequence of positive numbers that tends to zero. Suppose also that there exist $\mu, \nu > 0$, $C>1$, and $n_0 \in \N$ such that
\begin{equation}
\label{e1newversion}
 \frac{x_{n+1}}{x_{n+2}} \leq C^2 \left(\frac{x_n}{x_{n+1}}\right)^\nu, \ \text{for } n \geq n_0 \ \text{such that } x_n \geq x_{n+1},
\end{equation}
and
\begin{equation}
\label{e2newversion}
\frac{1}{C^2} \left(\frac{x_{n+1}}{x_{n}}\right)^\mu \leq \frac{x_{n+2}}{x_{n+1}}, \ \text{for } n \geq n_0 \ \text{such that } x_n \leq x_{n+1}.
\end{equation}
Then there exist $\alpha, \beta > 0$ such that
\begin{equation}
\label{eq:holdernew}
x_n^\alpha \leq x_{n+1} \leq x_n^\beta, \qfor \text{all sufficiently large values of } n.
\end{equation}
\end{proposition}
\begin{proof}
Let $N \geq n_0$ be sufficiently large that, for all $n \geq N$, we have both $x_n < 1$ and also
\begin{equation}
\label{N}
\frac{-2\log C}{\log x_n} < \frac{\mu}{2}.
\end{equation}
Set $p_n \defeq \dfrac{\log x_{n+1}}{\log x_n}$, for $n \geq N$, so that $x_{n+1} = x_n^{p_n}$ and $x_{n+2} = x_n^{p_np_{n+1}}$.

We first establish the upper bound in \eqref{eq:holdernew} by showing that $p_n \ge \beta  \defeq \frac{\mu}{2(\mu+1)}$ for $n\ge N$. Otherwise there exists $n\ge N$  such that $p_n < \beta <1$, in which case \eqref{e2newversion} holds. It would then follow that
\[ x_{n+2} \ge \frac{x_{n+1}}{C^2}\left(\frac{x_{n+1}}{x_n}\right)^\mu = \frac{x_n^{p_n(1+\mu)}}{C^2x_n^\mu} > \frac{x_n^{\mu/2}}{C^2x_n^\mu} > 1, \]
where the final inequality uses \eqref{N}. This contradicts the choice of $N$.

To complete the proof, we now take $n\ge N$ and seek an upper bound for $p_{n+1}$. For this $n$, either \eqref{e1newversion} holds or \eqref{e2newversion} holds, and we take $\lambda$ to be $\nu$ or $\mu$ respectively. Taking logarithms of \eqref{e1newversion} or \eqref{e2newversion} gives
\[ p_n(1-p_{n+1}) \log x_n \leq 2\log C + \lambda(1-p_n) \log x_n. \]
Dividing by $-p_n\log x_n > 0$ yields
\[ p_{n+1} - 1\leq \frac{-2\log C}{p_n\log x_n} + \lambda\left(\frac{p_n - 1}{p_n}\right). \]
Therefore, by \eqref{N} and the first part of the proof, we find that $p_{n+1} \le 1 + \frac{\mu}{2\beta} + \lambda$.
\end{proof}

To see that the reverse implication does not hold, let $x_0 = \frac{1}{2}$, and then set
\[
x_n \defeq \begin{cases} \frac{x_{n-1}}{2}, &\text{ for } n \in \N \text{ even}, \\
                    x_{n-1}^2, &\text{ for } n \in \N \text{ odd}.
			\end{cases}
\]
Then $x_n^2 \leq x_{n+1} \leq x_n$, for $n \geq 0$ and so \eqref{eq:holdernew} is satisfied. However, when $n$ is odd $\frac{x_{n+1}}{x_{n+2}} = \frac{1}{x_{n+1}}$ and $\frac{x_n}{x_{n+1}} = 2$, and so \eqref{e1newversion} cannot hold.

\mbox{ }

\noindent
\textbf{Acknowledgments.} We are grateful to John Osborne for helpful discussions.
\bibliographystyle{acm}
\bibliography{Orbits}
\end{document}